\documentclass{article}

\usepackage{algorithm}
\usepackage{algpseudocode}
\usepackage{amsthm}
\usepackage{amsmath}

\newtheorem{theorem}{Theorem}
\newtheorem{lemma}{Lemma}

\title{Computing optimal strategy for cop in the game of Cop v.s. Gambler}
\author{Shen-Fu Tsai \\ parity@gmail.com}

\begin{document}
\maketitle
\begin{abstract}
	We present two efficient algorithms that compute the optimal strategy for cop in the game of Cop v.s.  Gambler where the gambler's strategy is not optimal but known to the cop. The first algorithm is analogous to Bellman–Ford algorithm\cite{bellman_ford} for single source shortest path problem and runs in $O(|V(G)||E(G)|)$ time. The second is analogous to Dijkstra's algorithm\cite{dijkstra} and runs in $O(|E(G)|+|V(G)|\log |V(G)|)$ time. Compared with each other, they are more suitable for sparse and dense graphs, respectively.
\end{abstract}

\section{Introduction}
In the probabilistic version of game of graph pursuit\cite{komarov_arxiv_2013}, a cop plays against a gambler on a graph $G$. Before the game starts, the cop picks and occupies a vertex from $G$. In each round of the game, the cop selects and moves to a adjacent vertex or stays at the same vertex, and the gambler chooses to occupy a vertex randomly based on a time-independent distribution, or {\it gamble},  known to the cop. The gambler is not restricted to only adjacent vertices.  Whenever they occupy the same vertex at the same time the cop wins.\\

It is known that on a connected graph $G$, the cop can win in less than or equal to $n=|V(G)|$ expected rounds, so the cop {\it will} win in exactly $n$ expected rounds because the optimal gamble for the gambler is uniform distribution. However, we note that the cop's strategy proposed in \cite{komarov_arxiv_2013} that guarantees $n$ expected chase time may not be optimal when a non-uniform gamble is employed. An example is a chain $v_0, v_1, v_2, v_3$ with $(p_0,p_1,p_2,p_3)=(0.3,0.7,0,0)$ and the cop starting at $v_0$. By the strategy in \cite{komarov_arxiv_2013}, the cop should stay at $v_0$ and expect to win in $10/3$ time. However if she moves to $v_1$ and stay there, the expected time is $1+0.7/0.7=2$.\\

\section{Optimal Cop Strategy}
We describe two algorithms to compute the optimal strategy for cop starting at {\it every} vertex $v\in V(G)$ given $G$ and gamble $\{p_v:v\in V(G)\}$, where $p_v$ is the probability that the gambler selects to occupy vertex $v$. Our algorithms work for both directed and undirected graphs.\\

We can always represent the optimal strategy for cop when arriving vertex $v$ as to the next vertex $u\in N(v)$ to move to in the next round, where $N(v)$ is the set of adjacent vertices of $v$ plus $v$ itself.\\

\subsection{$O(|V(G)||E(G)|)$ Algorithm}
\label{sec:bellman_ford_algorithm}
Analogous to Bellman-Ford algorithm\cite{bellman_ford} for single source shortest path problem, in each of the $O(|V(G)|)$ iterations we update the strategy and chase time for vertex $u$ based on $v$ for each edge $(u,v)\in E(G)$. It runs faster on sparse graphs than the next algorithm in Section \ref{sec:dijkstra_algorithm}.\\

\begin{algorithm}
\begin{algorithmic}
	\ForAll {$v\in V(G)$}
		\State {$T_1(v)=1/p_v$, $\pi(v)=v$}
	\EndFor

	\For {$i=2,3,\ldots$}
		\State {$update=False$}
		\ForAll {$v\in V(G)$}
			\State {$T_i(v)=T_{i-1}(v)$}
			\If {$T_{i}(v)>1+(1-p_v)\min_{u\in N(v)}T_{i-1}(u)$}
				\State {$T_{i}(v)=1+(1-p_v)\min_{u\in N(v)}T_{i-1}(u)$}
				\State {$\pi(v)=\arg\min_{u\in N(v)}T_{k-1}(u)$}
				\State {$update=True$}
			\EndIf
		\EndFor

		\If {not $update$}
			\State {return}
		\EndIf
	\EndFor
\end{algorithmic}
\end{algorithm}

\subsubsection{Analysis}
In this section we show the correctness and efficiency of our algorithm. First define $T(v)$ as the
optimal expected chase time when the cop enters vertex $v$.\\

\begin{lemma}
	For each $v\in V(G)$,
	$$
	T(v)=1+(1-p_v)\min_{w\in N(v)}T(w).
	$$
	Moreover if $u\in N(v)$ and $T(v)=1+(1-p_v)T(u)$ then moving to $u$ after entering $v$ is optimal.
\end{lemma}

The following lemma says that our algorithm will never obtain any $T_{i}(v)$ that is smaller than
$T(v)$.
\begin{lemma}
	\label{lemma:achievable}
	For each $v\in V(G)$ and $i\geq 0$, $T_{i}(v)\geq T(v)$.
\end{lemma}
\begin{proof}
	We prove the lemma by showing that every $T_{i}(v)$ is {\it achievable}. For $i=0$ it is
	true, because $T_0(v)=1/p_v$ is the expected chase time for the cop who stays at vertex $v$
	forever. Assume $T_{k-1}(v)$ is achievable,
	$$
	T_{k}(v)=1+(1-p_v)\min_{u\in N(v)}T_{k-1}(u).
	$$
	By induction, $1+(1-p_v)\min_{u\in N(v)}T_{k-1}(u)$ on the right hand side is the upper
	bound of chase time if the cop moves from vertex $v$ to $\arg\min_{u\in N(v)}T_{k-1}(u)$ when she doesn't capture the
	gambler in the current round. Therefore $T_{k}(v)$ is achievable too.
\end{proof}

Define a {\it chase path} as a path $\{v_0,v_1,\ldots,v_k\}$ with
$$
T(v_i)=1+(1-p_{v_i})T(v_{i+1}), v_{i+1}\in N(v_i)
$$
for each $i\in [0,k-1]$ and
$$
T(v_k)=1+(1-p_{v_k})T(v_{k})
$$
Clearly a chase path starting from $v$ and ending at $u$ is an optimal path for cop to start
chasing the gambler at $v$. Ending at $u$ implies that it is optimal for cop to stay at $u$
forever.\\

\begin{lemma}
	If $\{v_0,v_1,\ldots,v_k\}$ is a chase path, then so is $\{v_1,\ldots,v_k\}$.
\end{lemma}

A chase path starting from $v$ is a {\it shortest chase path} if these is no shorter chase path that starts from $v$. We say a vertex $v$ has {\it shortest chase length} $k$ if the shortest chase path starting from $v$ has length $k$.

\begin{lemma}
	\label{lemma:prefix}
	If $\{v_0,v_1,\ldots,v_k\}$ is a shortest chase path, then so is $\{v_1,\ldots,v_k\}$.
\end{lemma}

\begin{lemma}
	\label{lemma:nei}
	If $v$ has shortest chase length $k>1$, then it has a neighbor $u\neq v$ with shortest chase length
	$k-1$ such that $u$ follows $v$ in the shortest chase path.
\end{lemma}
\begin{proof}
	Let path $\{v_0=v,v_1,\ldots,v_{k-1}\}$ be a shortest chase path. By Lemma
	\ref{lemma:prefix}, $\{v_1,\ldots,v_{k-1}\}$ is a shortest chase path of length $k-1$. By
	definition $v_1\in N(v)$, therefore $u=v_1$.
\end{proof}

\begin{lemma}
	\label{lemma:induction}
	If $v$ has shortest chase length $k$, then $T_k(v)=T(v)$ and the computed $\pi(v)\in N(v)$ remains
	unchanged thereafter satisfying
	$$
	T(v)=1+(1-p_v)T(\pi(v)).
	$$
\end{lemma}
\begin{proof}
	For $k=1$ the statement holds because if a vertex has shortest chase length $1$, then
	$$
	p_v=\max_{u\in N(v)}p_u
	$$
	as otherwise moving to a neighbor $w$ of $v$ with higher probability is better than staying
	at $v$ for good.\\

	For a vertex $v$  with shortest chase length $k>1$, by Lemma \ref{lemma:nei} it has a
	neighbor $u\neq v$ with shortest chase length $k-1$. By induction assumption
	$T_{k-1}(u)=T(u)$,
	\begin{align*}
		T_{k}(v)&=1+(1-p_v)\min_{w\in N(v)}T_{k-1}(w) \\
			&\leq 1+(1-p_v)T_{k-1}(u)=1+(1-p_v)T(u) \\
			&=T(v)
	\end{align*}
	The last equality stems from the fact that $u$ follows $v$ in the shortest chase path.
	On the other hand, by Lemma \ref{lemma:achievable}, $T_k(v)\geq T(v)$, so $T_k(v)=T(v)$. It
	is obvious from the description of our algorithm that $\pi(v)$ satisfies
	$$
	T(v)=1+(1-p_v)T(\pi(v)).
	$$
	Moreover $\pi(v)$ is only updated in round $j+1$ when $T_{j+1}(v)<T_j(v)$, so as
	$T_{k}(v)=T_{k+1}(v)=\ldots$, $\pi(v)$ does not change beyond round $k$.
\end{proof}


\begin{lemma}
	\label{lemma:simple_path}
	A shortest chase path is simple.
\end{lemma}
\begin{proof}
	Suppose on the contrary, that a shortest chase path $P$ passes some vertex $v$ twice. Then
	deleting the first occurrence of $v$ up to but excluding its second occurrence gives
	another chase path which is shorter than $P$.
\end{proof}
\begin{theorem}
	The algorithm in Section \ref{sec:bellman_ford_algorithm} correctly computes $\pi(v)$ for all $v\in
	V(G)$ in $O(|V(G)||E(G)|)$ time.
\end{theorem}
\begin{proof}
	By Lemma \ref{lemma:simple_path}, all shortest chase paths have lengths less than or equal
	to $n$. So by Lemma \ref{lemma:induction} after round $n$ the computed $\pi(v)$ satisfies 
	$$
	T(v)=1+(1-p_v)T(\pi(v))
	$$
	and $T_n(v)=T_{n+1}(v)=T_{n+2}(v)=\ldots$, i.e. the algorithm terminates no later than
	round $n+1$. Each round takes $O(|E(G)|)$ time, so the overall time complexity is
	$O(|V(G)||E(G)|)$.
\end{proof}

\subsection{$O(|E(G)|+|V(G)|\log |V(G)|)$ Algorithm}
\label{sec:dijkstra_algorithm}

We present another algorithm computing optimal cop strategy. The time complexity is $O(|E(G)|+|V(G)|\log |V(G)|)$, so it is more suitable for dense graphs. Analogous to Dijkstra's algorithm\cite{dijkstra} for single source shortest path, in each of the $|V(G)|$ iterations we only update the strategy and chase time of vertices that links directly to a specific vertex.

\begin{algorithm}
\begin{algorithmic}
	\ForAll {$v\in V(G)$}
		\State {$t(v)=1/p_v$, $\pi(v)=v$}
	\EndFor

	\State {$S=V(G)$}
	\While {$|S|>0$}
		\State {$u=\arg\min_{v\in S}t(v)$}
		\State {$S=S-u$}
		\ForAll {$w\in S$ such that $(w,u)\in E(G)$}
			\If {$t(w)>1+(1-p_w)t(u)$}
				\State {$t(w)=1+(1-p_w)t(u)$}
				\State {$\pi(w)=u$}
			\EndIf
		\EndFor
	\EndWhile
\end{algorithmic}
\end{algorithm}

\subsubsection{Analysis}
In this section we show the correctness and efficiency of this algorithm.\\

\begin{lemma}
	\label{lemma:equality}
	If $u$ follows $v$ in a chase path and $T(u)=T(v)$, then $p_u=p_v$ and $T(u)=T(v)=1/p_u$.
\end{lemma}

A chase path always goes from long to short expected chase time.
\begin{lemma}
	\label{lemma:long_to_short}
	If $v$ follows $u$ in a chase path, then $T(v)\geq T(u)$.
\end{lemma}
\begin{proof}
	Suppose otherwise. Since $v$ follows $u$, $T(v)=1+(1-p_v)T(u)$, or $T(u)=(T(v)-1)/(1-p_v)>T(v)$. So $T(v)>1/p_v$, a contradiction.
\end{proof}

\begin{lemma}
	\label{lemma:alt_correctness}
	For every $u\in V(G)-S$, $t(u)=T(u)$.
\end{lemma}
\begin{proof}
	It clearly holds when $|S|=n$. Suppose the statement holds for $|S|=k$, and assume $|S|=k-1>0$ at the beginning of some iteration of the while loop. Let $w=\arg\min_{u\in S}t(u)$. It suffices to show that $t(w)=T(w)$.\\

	Assume $t(w)>T(w)$. There exists a chase path $P=\{w,u_0,\ldots,u_m\}$. If $P$ and $V(G)-S$ are disjoint, then by Lemma \ref{lemma:long_to_short}, $T(w)=T(u_0)=\ldots=T(u_m)$. Since $u_m$ is the end of the chase path $T(u_m)=1/p_{u_m}$, a contradiction as $t(u_m)\leq 1/p_{u_m}$. If $v\in P$ is the vertex closest to $w$ in $P$ such that $v\in V(G)-S$, let $u$ follows $v$ in $P$ and so $u\in S$. By Lemma \ref{lemma:long_to_short} $T(u)=T(w)$. If $u=w$, then $t(w)$ has already been updated to $1+(1-p_w)t(v)=1+(1-p_w)T(v)=T(w)$. If $u\neq w$, then by Lemma \ref{lemma:equality}, $T(w)=1/p_w\geq t(w)$.\\

	We are then done here, because the computed $t(w)$ is always achievable, i.e. $t(w)\geq T(w)$.
\end{proof}

\begin{theorem}
	The algorithm correctly computes $T(v)$ and optimal strategy for all $v\in
	V(G)$ in $O(|E(G)|+|V(G)|\log |V(G)|)$ time. 
\end{theorem}
\begin{proof}
	The correctness follows Lemma \ref{lemma:alt_correctness} immediately. To achieve $O(|E|+|V(G)|\log |V(G)|)$ time complexity, $S$ could be kept as a Fibonacci heap\cite{fibonacci_heap} with $O(1)$ amortized element update time and $O(\log |V(G)|)$ element removal time.
\end{proof}


\begin{thebibliography}{}
	\bibitem{komarov_arxiv_2013}
		Natasha Komarov and Peter Winkler, Cop vs. Gambler, Discrete Math. 339 (2016),
		1677-1681.
	\bibitem{fibonacci_heap}
		https://en.wikipedia.org/wiki/Fibonacci\_heap
	\bibitem{bellman_ford}
		https://en.wikipedia.org/wiki/Bellman
	\bibitem{dijkstra}
		https://en.wikipedia.org/wiki/Dijkstra
\end{thebibliography}
\end{document}